\documentclass[reqno]{amsart}
\usepackage{amssymb} 
\usepackage{amscd} 

\newtheorem{theorem}{Theorem}[section]
\newtheorem{lemma}[theorem]{Lemma}

\newtheorem{remark}[theorem]{Remark}
\newtheorem{proposition}[theorem]{Proposition}
\newtheorem{corollary}[theorem]{Corollary}

\newcommand{\F}{\boldsymbol{F}}
\newcommand{\FF}{\overline{\boldsymbol{F}}}
\newcommand{\TT}{\boldsymbol{\mathcal{T}} }

\newcommand{\VV}{\boldsymbol{\mathcal{V}} }

\DeclareMathOperator{\calbV} {\boldsymbol{\mathcal V}}

\usepackage{color}  
\definecolor{darkgreen}{rgb}{0.03, 0.5, 0.03}

\newcommand{\co} {\boldsymbol c}

\newcommand{\Spec} {\mbox{\rm Spec}}
\newcommand{\QMax} {\mbox{\rm QMax}}
\newcommand{\QSpec} {\mbox{\rm QSpec}}

\newcommand{\Fb}{\boldsymbol{\overline{F}}}
\newcommand{\f}{\boldsymbol{{f}}}

\newcommand{\stf} {\star{_{\!{_f}}}}
\newcommand{\stt} {\widetilde{\star}}

\newcommand{\bo}{\bigstar_0}
\newcommand{\ba}{\bigstar}
\newcommand{\sus}{\blacktriangle^{\!\star}}
\newcommand{\susT}{\blacktriangle^{\!\star}_T}
\newcommand{\susK}{\blacktriangle^{\!\star}_K}

\newcommand{\sust}{\blacktriangle^{\!\stt}}
\newcommand{\susf}{\blacktriangle^{\!\stf}}

\newcommand{\Bast} {\Large\mbox{$\boldsymbol\ast$}}
\newcommand{\bast} {\large\mbox{$\boldsymbol\ast$}}

\begin{document}
\title[]{Polynomial extensions of semistar operations}

\author{Gyu Whan Chang, Marco Fontana, and Mi Hee Park}


\thanks{\it Acknowledgments. \rm
During the preparation of this paper,  the first-named author
was supported by the University of Incheon Research Grant in 2011
(Grant No. 2011-0288),  the  second-named author was  partially
supported by  a research grant PRIN-MiUR  2008WYH9NY\_003, and the third-named
author was supported by Basic Science Research Program through the
National Research Foundation of Korea (NRF) funded by the Ministry
of Education, Science and Technology (2010-0021883).}
\address{Department of Mathematics, University of Incheon,
Incheon 406-772, Korea}
\email{whan@incheon.ac.kr}
\address{Dipartimento di Matematica, Universit\`a degli Studi
``Roma Tre'', 00146 Rome, Italy.}
\email{fontana@mat.uniroma3.it }
\address{Department of Mathematics, Chung-Ang University,
Seoul 156-756, Korea}
\email{mhpark@cau.ac.kr}


\maketitle

\begin{abstract}
We provide a complete solution to the problem of extending
arbitrary semistar operations of an integral domain $D$ to
semistar operations of the polynomial ring $D[X]$.  As an
application, we show that one can reobtain the main results of the
papers \cite{CF} and \cite{CF2011} concerning the problem in the
special cases of stable semistar operations of finite type or
semistar operations defined by families of overrings. Finally, we
investigate the behavior of the polynomial extensions of the most
important and classical operations such as $d_D$, $v_D$, $t_D$,
$w_D$ and $b_D$ operations.
\end{abstract}

\bigskip

\section{Preliminaries}

Let $D$ be an integral domain with quotient field $K$. Let
$\boldsymbol{\overline{F}}(D)$ denote the set of all nonzero
$D$-submodules of $K$ and let $\boldsymbol{F}(D)$ be the set of
all nonzero fractional ideals of $D$, i.e., $E \in
\boldsymbol{F}(D)$ if $E \in \boldsymbol{\overline{F}}(D)$ and
there exists a nonzero $d \in D$ with $dE \subseteq D$. Let
$\f(D)$ be the set of all nonzero finitely generated
$D$-submodules of $K$. Then, obviously, $\f(D) \subseteq
\boldsymbol{F}(D) \subseteq \boldsymbol{\overline{F}}(D)$.

Following Okabe-Matsuda \cite{o-m}, a \emph{semistar operation} of
$D$ is a map $\star: \boldsymbol{\overline{F}}(D) \to
\boldsymbol{\overline{F}}(D),\ E \mapsto E^\star$, such that, for
all $x \in K\setminus \{0\}$ and for all $E,F \in
\boldsymbol{\overline{F}}(D)$, the following properties hold:
\begin{enumerate}
\item[$(\star_1)$] $(xE)^\star=xE^\star$; \item[$(\star_2)$] $E
\subseteq F$ implies $E^\star \subseteq F^\star$;
\item[$(\star_3)$] $E \subseteq E^\star$ and $E^{\star \star} :=
\left(E^\star \right)^\star=E^\star$.
\end{enumerate}

A \emph{(semi)star operation} is a semistar operation that,
restricted to $\boldsymbol{F}(D)$,  is a star operation (in the
sense of \cite[Section 32]{G}). It is easy to see that a
semistar operation $\star$ of $D$ is a (semi)star operation if and
only if $D^\star = D$.

If $\star$ is a semistar operation of $D$, then we can consider a
map\ $\star_{\!_f}: \boldsymbol{\overline{F}}(D) \to
\boldsymbol{\overline{F}}(D)$ defined as follows:
$$E^{\star_{\!_f}}:=\bigcup \{F^\star\mid F \in \boldsymbol{f}(D)
\mbox{ and } F \subseteq E\}  \; \mbox{ for each } E \in
\boldsymbol{\overline{F}}(D).$$

\noindent It is easy to see that $\star_{\!_f}$ is a semistar
operation of $D$, which is called \emph{the semistar operation of
finite type associated to $\star$}.  Note that, for each $F \in
\boldsymbol{f}(D)$, $F^\star=F^{\star_{\!_f}}$.  A semistar
operation $\star$ is called a \emph{semistar operation of finite
type} if $\star=\star_{\!_f}$.

If $\star_1$ and $\star_2$ are two semistar operations of $D$ such
that $E^{\star_1} \subseteq E^{\star_2}$ for all $E \in \FF(D)$,
then we say that $\star_1 \leq \star_2$. This is equivalent to say
that $\left(E^{\star_{1}}\right)^{\star_{2}} = E^{\star_2}=
\left(E^{\star_{2}}\right)^{\star_{1}}$ for each $E \in \Fb(D)$.
Obviously, for each semistar operation $\star$ of $D$, we have
$\star_{\!_f} \leq \star$. Let $d_D$ (or, simply, $d$) be the {\it
identity semistar operation of $D$}; clearly $d \leq \star$ for
all semistar operations $\star$ of $D$. Let $e_D$ (or, simply,
$e$) be the {\it trivial semistar operation of $D$}, defined by
$E^e = K$ for each $E\in \FF(D)$; clearly $\star \leq e$ for all
semistar operations $\star$ of $D$.

Let $\star$ be a semistar operation of $D$. We say that a nonzero
integral ideal $I$ of $D$ is a \emph{quasi-$\star$-ideal} if
$I^\star \cap D = I$, a \emph{quasi-$\star$-prime  ideal} if it is
a prime quasi-$\star$-ideal, and a \emph{quasi-$\star$-maximal
ideal} if it is maximal in the set of all proper
quasi-$\star$-ideals. A quasi-$\star$-maximal ideal is  a prime
ideal. It is  easy to prove that each  proper quasi-$\star_{_{\!
f}}$-ideal is contained in a quasi-$\star_{_{\! f}}$-maximal
ideal.  More details can be found in \cite[page 4781]{FL}. We will
denote by $\QMax^{\star}(D)$  (respectively, $\QSpec^\star(D)$)
the set of all quasi-$\star$-maximal ideals  (respectively,
quasi-$\star$-prime ideals) of $D$. When $\star$ is a (semi)star
operation, the notion of  quasi-$\star$-ideal coincides with the
``classical'' notion of  {\it   integral $\star$-ideal} (i.e., a
nonzero  integral ideal $I$ such that $I^\star = I$).

If $\Delta$ is a set of prime ideals of $D$, then the semistar
operation $\star_\Delta$  defined by
$$ E^{\star_\Delta} := \bigcap \{ED_P \;|\;\, P \in \Delta\}\, \;
\textrm {  for each}    \; E \in \boldsymbol{\overline{F}}(D)\,
$$

\noindent is called {\it the spectral semistar operation of $D$
associated to } $\Delta$. A semistar operation $\star$ of $D$ is
called {\it a spectral semistar operation} if there exists a
subset $ \Delta$ of the prime spectrum of $D$, $\mbox{\rm
Spec}(D)$,  such that $\,\star = \star_\Delta\,$.

A semistar operation $\star$ is \emph{stable} if $(E \cap F)^\star
= E^\star \cap F^\star$ for each $E,F \in \Fb(D)$. Clearly,
spectral semistar operations are stable \cite[Lemma
4.1(3)]{FH2000}.

When $\star$ is a semistar operation of $D$ and $\Delta :=
\QMax^{\star_{_{\! f}}}(D)$, we set $\stt:= \star_{\Delta}$, i.e.,
$$E^{\stt} := \bigcap \left \{ED_P \mid P \in
\QMax^{\star_{_{\! f}}}(D) \right\}  \;  \mbox{ for each } E \in
\boldsymbol{\overline{F}}(D),$$ \noindent or equivalently,
$$E^{\stt} = \bigcup \{(E:J) \mid J \in \f(D),\ J
\subseteq D, \mbox{ and } J^\star =D^\star \} \; \mbox{ for each }
E \in \boldsymbol{\overline{F}}(D).$$ \noindent Then $\stt\,$  is
a stable semistar operation of finite type, which is called
\emph{the stable semistar operation of finite type associated to
$\star$}. 
It is known that if a semistar operation $\star$ is stable and of
finite type, then $\star = \stt\,$ \cite[Corollary
3.9(2)]{FH2000}.

By $v_D$ (or, simply, by $v$) we denote  the $v$--semistar
operation of $D$ defined as usual, that is, $E^v := (D:(D:E))$ for
each $E\in \F(D)$ and $E^v:=K$ for each $E\in \FF(D)\setminus
\F(D)$. By $t_D$ (or, simply, by $t$) we denote  $(v_D)_{_{\! f}}$
the semistar operation of finite type associated to $v_D$ and by
$w_D$ (or just by $w$) the stable semistar operation of finite
type associated to $v_D$ (or, equivalently, to $t_D$), considered
by F.G.  Wang  and R.L. McCasland in \cite{WMc97} (cf. also
\cite{gv});  i.e.,  $w_D := \widetilde{\ v_D} = \widetilde{\
t_D}$.  Clearly $w_D\leq t_D \leq v_D$.  Moreover, it is easy to
see that for each   (semi)star operation $\star$ of $D$, we have
$\star \leq v_D$, $\stf \leq t_D$, and $\stt\leq w_D$ (cf. also
\cite[Theorem 34.1(4)]{G}).

Let $\star$ be a semistar operation of $D$   and let $F\in
\boldsymbol{f}(D)$. We say that $F$ is {\it
$\star$--\texttt{eab}} (respectively, {\it $\star$--\texttt{ab}})
if,   for all $G,\ H \in \boldsymbol{f}(D)$ (respectively, for
all $G,\ H \in \overline {\boldsymbol{F}}(D)$),  $(FG)^{\star}
\subseteq (FH)^{\star}$ implies $G^{\star}\subseteq H^{\star}$.
The operation $\star$ is said to be {\it \texttt{eab}}
(respectively, {\it \texttt{ab}} ) if each $F\in
\boldsymbol{f}(D)$ is $\star$--\texttt{eab} (respectively,
$\star$--\texttt{ab}). \  An \texttt{ab} operation is obviously an
\texttt{eab} operation.

Using the fact that, given $F \in \f(D)$, $F$ is
$\star$--\texttt{eab} if and only if  $\left( (FH)^\star : F^\star
\right) = H^\star$ for each $H \in \boldsymbol{f}(D)$ \cite[Lemma
8]{FL:2009}, we can associate to each semistar operation
$\,\star\,$ of $D$  an \texttt{eab} semistar operation  $\,
\star_a\, $ of finite type, which is called {\it the \texttt{eab}
semistar operation associated to $\,\star\,$} and defined as
follows: for each $ F \in \boldsymbol{f}(D)$ and for each  $E \in
{\overline{\boldsymbol{F}}}(D)$,
$$
\begin {array} {rl}
F^{\star_a} :=& \hskip -5pt  \bigcup\{((FH)^\star:H^\star) \; \ |
\; \, H \in
\boldsymbol{f}(D)\}\,, \\
E^{\star_a} :=& \hskip -5pt  \bigcup\{F^{\star_a} \; | \; \, F
\subseteq E\,,\; F \in \boldsymbol{f}(D)\}
\end{array}
$$
\cite[Definition 4.4 and Proposition 4.5]{FL1}. \rm  The previous
construction, in the ideal systems setting, is essentially due to
P. Jaffard \cite{J} and {F. Halter-Koch} \cite{HK:1997},
\cite{HK}.
Obviously,  $(\star_{_{\!f}})_{a}= \star_{a}$. Moreover, when
$\star = \star_{_{\!f}}$, then $\star$ is \texttt{eab} if and only
if $\star = \star_{a}\,$ \cite[Proposition 4.5(5)]{FL1}.  We call
the semistar operation $ b_D := (d_D)_{a}$ {\it the
$b$--operation} of $D$.  It is easy to see that $b_D$ is a
(semi)star operation of $D$ if and only if $D$ is integrally
closed.

Given a family of semistar operations $\{\star_\lambda\mid \lambda
\in \Lambda\}$ of $D$, the semistar operation $\wedge
\star_\lambda$ of $D$ is defined by
$$
E^{\wedge \star_\lambda} := \bigcap  \{ E^{\star_\lambda} \mid
\lambda \in \Lambda \} \;\; \mbox{ for each } E \in \Fb(D).
$$

Let $\TT := \{T_{\lambda}  \mid \lambda \in \Lambda \}$ be a set
of overrings of $D$. We denote by $\star_{\{T_\lambda\}}$ the
semistar operation of $D$ defined by $E^{\star_{\{T_\lambda\}}} :=
ET_\lambda$ for each $E \in \FF(D)$ and by $\wedge_{\TT}$ the
semistar operation $\wedge \{\star_{\{T_\lambda\}} \mid \lambda
\in \Lambda\}$.

For a semistar operation  $\,\star\,$ of $\,D\,$, \ we say that a
valuation overring $\,V\,$ of $\,D\,$ is \it a $\star$--valuation
overring of $\,D\,$ \rm provided $\,F^\star \subseteq FV\,$ (or,
equivalently, $\,F^\star V =FV$) \ for each $\,F \in
\boldsymbol{f}(D)\,$. Let $\calbV(\star)$ be the family of all
$\star$--valuation overrings of $D$. Then the semistar operation
$\wedge_{\calbV(\star)}$ of $D$ is an \texttt{ab} semistar
operation \cite[page 2098]{FL:2009}; clearly,
$\wedge_{\calbV(\star)}= \wedge_{\calbV(\star_f)}$.
Note that
$$
\star_a =\wedge_{\calbV(\star)}\,,\; \mbox{ in particular, } \,
b_D  = \wedge_{\calbV(d_D)}\,.$$  This result follows from
\cite[Proposition 4.1(5)]{FL}.

\smallskip

We now consider the polynomial ring $D[X]$ over $D$. A semistar operation
$\ba$ of $D[X]$ is called an {\it extension} (respectively, a {\it
strict extension}) {\it of a semistar operation} $\star$   of $D$
if $ E^{\star} = E[X]^{\ba} \cap K$ (respectively, $ E^{\star}[X]
= E[X]^{\ba}$) for all $E \in \Fb(D)$.

Given a semistar operation $\ba$ of $D[X]$, set $E^{\bo} :=
(E[X])^{\bigstar} \cap K$ for each $E \in \FF(D)$. Then $\bo$ is a
semistar operation of $D$ and $\ba$ is an extension of $\bo$. By
\cite[Lemma 5]{CF2011}, $({\ba}_{_{\! f}})_0 = (\bo)_{_{\! f}}$
and $(\widetilde{\ba})_0 =\widetilde{\ \bo\ }$. It is easy to see
that $(d_{D[X]})_{_{\!{0}}} = d_{D}$ and $(v_{D[X]})_{_{\!{0}}} =
v_{D}$, and therefore, $(t_{D[X]})_{_{\!{0}}} = t_{D}$ and
$(w_{D[X]})_{_{\!{0}}} = w_{D}$.
In fact, it
is known that:
$$
\begin{array}{rl}
(E[X])^{v_{D[X]}}= &\hskip -6pt E^{v_D}[X] \; \mbox{ for all $E\in
\F(D)$, } \\(E[X])^{t_{D[X]}}=&\hskip -6pt E^{t_D}[X] \mbox{ \;
and \; } (E[X])^{w_{D[X]}}=E^{w_D}[X] \; \mbox{ for all $E \in
\FF(D)$. }
\end{array}
$$
Thus $t_{D[X]}$ and $w_{D[X]}$ are strict extensions of $t_D$ and
$w_D$, respectively.

\smallskip

The present work is devoted to the following problem: how to
extend ``in a canonical way'' an {\sl arbitrary} semistar
operation of $D$ to the polynomial ring $D[X]$. The first attempts
to extend a semistar operation of $D$ to a semistar operation of
$D[X]$ were done by G. Picozza \cite{Pi} and then by the first two
authors of this paper \cite{CF, CF2011}.
Their study is focused on the stable semistar operations of finite
type.  In this paper, we provide a complete solution to this
problem in the most general setting. As an application, we show
that, in the particular cases of stable semistar operations of
finite type or semistar operations defined by families of
overrings, we reobtain the main results given in \cite{CF} and
\cite{CF2011}.
Finally, we investigate the behavior of the polynomial extensions
of  some operations among the most important and classical ones
such as the $d_D$, $v_D$, $t_D$, $w_D$ and $b_D$ operations.

\smallskip

To be more precise, in Section 2, we show that there always exists
the maximum in the   set of all strict extensions to the
polynomial ring $D[X]$  of a given semistar operation $\star$ on
$D$. Let $\blacktriangle^{\!\star}$ denote this semistar operation
of $D[X]$. After giving an explicit description of
$\blacktriangle^{\!\star}$, we show that such a semistar operation
is never of finite type and we investigate the properties of
$\susf$ and $\sust$; in particular,
$(\sus)_{\!{_{f}}}=(\susf)_{\!{_{f}}}$ is the largest finite-type
strict extension of $\stf$ and
$\widetilde{\sus}=\widetilde{\blacktriangle^{\tilde{\star}}}$ is
the largest stable finite-type strict extension of
$\tilde{\star}$. As an application, we consider some of the
classical operations and we prove that
$({\blacktriangle^{v_D}})_{\!{_{f}}}=({\blacktriangle^{t_D}})_{\!{_{f}}}=t_{D[X]}$
and
$\widetilde{\blacktriangle^{v_D}}=\widetilde{\blacktriangle^{w_D}}=w_{D[X]}$.
Moreover, for the trivial operations, we have
${\blacktriangle}^{e_D} \lneq e_{D[X]}$ and $d_{D[X]}\leq
({\blacktriangle}^{d_D})_f$, with
$d_{D[X]}=({\blacktriangle}^{d_D})_f$ if and only if $D$ is a
field.


After having observed  that each semistar operation  $\star$ of
$D$ admits infinitely many  strict extensions to $D[X]$ and, among
them, the largest one is $\blacktriangle^{\!\star}$, in Section 3
we show the existence of the smallest strict extension to $D[X]$.
Unlike the largest strict extension, the smallest strict
extension, denoted by ${\boldsymbol{\curlywedge}^\star}$, is not
in general described in an explicit form. However, in case of
stable semistar operations of finite type, we prove that
$\boldsymbol{\curlywedge}^{\stt}=[\stt]$, where $ A^{[\stt]} :=
\bigcap \{ AD_Q[X] \mid Q \in \QMax^{\stf}(D)\}$ for each $A \in
\FF(D[X])$.

In the last section, we generalize some results concerning the
polynomial extensions of a stable finite-type semistar operation
to the polynomial extensions of a semistar operation
defined by a given family of overrings of $D$.
As an application of the main result of the section, we obtain
$\boldsymbol{\curlywedge}^{b_D} \leq b_{D[X]}=
(\boldsymbol{\curlywedge}^{b_D})_a \leq
(\blacktriangle^{b_D})_a\lneq \blacktriangle^{b_D}$,
with $b_{D[X]}=(\blacktriangle^{b_D})_a$ if and only if $D$ is a field.


\section{Polynomial Strict Extensions of General Semistar Operations}


The goal of the present section is to define in  a canonical way
an extension to the polynomial ring $D[X]$ of a given semistar
operation  $\star$ of $D$.

\smallskip

For $A\in \FF(D[X])$ with $A \subseteq K[X]$, we denote by
$\co_{D}(A)$ 
the $D$-submodule of $K$
generated by the contents $\co_{D}(f)$ for all $f \in A$,
i.e., $\co_{D}(A) := \sum_{f \in A}\,\co_D(f)$.
Then, obviously, $ A \subseteq \co_{D}(A)[X]$.

\begin{theorem}\label{strictext}
Let $D$ be an integral domain with quotient field $K$,  let
$\star$ be a semistar operation of $D$, and let $T$ be an overring
of $D$ such that $T^\star = T$. For each $A \in \FF(D[X])$, set

$$
A^{\susT}:= \left\{
\begin{array}{ll}
\bigcap \{z^{-1} (\co_D(zA))^\star[X] \mid 0\neq z \in (T[X] :
A)\}\,, & \mbox{ if }
(T[X] : A) \neq (0)\,, \\
K(X)\,,&  \mbox{ if }
 (T[X] : A) =(0)\,. \\
\end{array}
\right.
$$
Then:
\begin{enumerate}
\item $\susT$ is a semistar operation of $D[X]$.
\item If $\star$ is a (semi)star operation of $D$, then
$\susT$ is a (semi)star operation of $D[X]$.

\item If $T' \subseteq T''$ are two overrings of $D$
such that $(T')^\star = T'$ and $ (T'')^\star =T''$, then
${\blacktriangle^{\!\star}_{T^{\prime}}} \geq
{\blacktriangle^{\!\star}_{T^{\prime\prime}}}$.

\item $\sus:=\susK$  is a strict extension of $\star$
to $D[X]$.  In fact, it  is the largest strict extension of
$\star$ to $D[X]$.
\end{enumerate}
\end{theorem}
\begin{proof} From the definition, it follows immediately that $A^{\susT} \in \FF(D[X])$.

{\bf Claim 1.} {\sl  Let $E\in \FF(D)$. If $E \subseteq T$, then
$(E[X])^{\susT}=E^\star[X] =(E^\star[X])^{\susT}$}.

Since $\co_D(E[X])=E$ and $1\in (T[X]:E[X])$, we have
$(E[X])^{\susT}\subseteq E^\star[X]$. For the opposite inclusion,
let $z \in (T[X]:E[X])\setminus (0)$. Then, obviously, $z\in
K[X]$. Write $z=z_0+z_1X+\cdots +z_nX^n$, with $z_i \in K$. Then
$\co_D(z E[X])=z_0E+z_1E+\cdots +z_nE$. Therefore, $z^{-1}
(\co_D(z E[X]))^\star[X]= z^{-1}(z_0E+z_1E+\cdots
+z_nE)^\star[X]\supseteq z^{-1}(z_0E^\star+z_1E^\star+\cdots
+z_nE^\star)[X] \supseteq z^{-1} (z E^\star[X])=E^\star[X]$. Thus,
we have $(E[X])^{\susT}\supseteq E^\star[X]$ and hence
$(E[X])^{\susT}=E^\star[X]$. Also, since $E^\star\in \FF(D)$ and
$E^\star \subseteq T^\star=T$, we have
$(E^\star[X])^{\susT}=(E^\star)^\star[X]=E^\star[X]$.

{\bf Claim 2.} {\sl For each $\alpha \in K(X)\setminus (0)$ and
$A \in \FF(D[X])$, $(\alpha A)^{\susT}=\alpha
A^{\susT}$.}

It follows from the fact that $(T[X]:A)=\alpha^{-1}(T[X]:\alpha
A)$.

{\bf Claim 3.} {\sl If  $A_1, A_2 \in \FF(D[X])$ and $A_1\subseteq
A_2$, then $A_1^{\susT}\subseteq A_2^{\susT}$. }

This is a straightforward consequence of the definition.

{\bf Claim 4.} {\sl For each $A \in \FF(D[X])$, $A\subseteq
A^{\susT}$.}

Let $z\in (T[X]:A)\setminus (0)$. Then $zA \subseteq \co_D(zA)[X]
\subseteq (\co_D(zA))^\star[X]$, and hence $A \subseteq z^{-1}
(\co_D(zA))^\star[X]$. Therefore, $A\subseteq A^{\susT}$.

{\bf Claim 5.} {\sl For each $A \in \FF(D[X])$,
$(A^{\susT})^{\susT}=A^{\susT}$.}

From Claims 3 and 4,  $A^{\susT}\subseteq (A^{\susT})^{\susT}$.
For the opposite inclusion, we may assume that $A^{\susT}\neq
K(X)$. Let $z \in (T[X]:A)\setminus (0)$.
By Claims 1, 2, and 3, we have $ z (A^{\susT})^{\susT} = (z
A^{\susT})^{\susT} =((z A)^{\susT})^{\susT}\subseteq
((\co_D(zA))^\star[X])^{\susT}=((\co_D(z
A))^\star)^\star[X]=(\co_D(z A))^\star[X]$, i.e.,
$(A^{\susT})^{\susT}\subseteq z^{-1}(\co_D(z A))^\star[X]$. Since
$z$ is an arbitrary nonzero element of $(T[X]:A)$, we have
$(A^{\susT})^{\susT}\subseteq A^{\susT}$.

\smallskip

(1) Claims 2--5 show that ${\susT}$ is a semistar operation
of $D[X]$.

(2)  Note that, by Claim 1, $D[X]^{\susT} = D^\star[X]$.

(3) is a direct consequence of the definition.

(4) By (1) and Claim 1, $\susK$ is a strict extension of
$\star$ to $D[X]$. In order to show that $\susK$ is the largest
strict extension of $\star$, let $\bigstar$ be a strict extension
of $\star$ to $D[X]$ and let $A \in \FF(D[X])$. If $(K[X]:A)
=(0)$, then clearly $A^\bigstar \subseteq K(X) = A^{\susK}$.
Assume that $(K[X]:A)\neq (0)$ and let $z \in (K[X]:A)\setminus
(0)$. Then $zA \subseteq \co_D(zA)[X]$ and so $z A^\bigstar=(zA)^\bigstar
\subseteq (\co_D(zA)[X])^\bigstar = (\co_D(zA))^\star[X]$, i.e.,
$A^\bigstar \subseteq z^{-1}(\co_D(zA))^\star[X]$. Hence
$A^\bigstar \subseteq A^{\susK}$.
%
\end{proof}

\begin{remark}
\rm For each $E\in \FF(D)$, set
$$
E^{\star_e}:= \left\{
\begin{array}{ll} E^\star\,,& \mbox{ if }
(D^\star : E) \neq (0)\,, \\
K\,, & \mbox{ otherwise. }
\end{array}
\right.
$$
Then ${\star_e}$ is a semistar operation of $D$ with $\star \leq
\star_e$ and $(D^\star)^{\star_e}=D^\star$. Using Claims 1 and 2,
we can easily show that $\blacktriangle^{\!\star}_{D^{\star}}$ is an extension of $\star_e$ to $D[X]$.
\end{remark}

\smallskip

In the following,
we investigate the properties of $\susf$ and $\sust$.  Note first
that, from the definition of $\sus$, it follows immediately that
$\star' \leq \star''$ implies that
${\blacktriangle^{\!\star^{\SMALL{\prime}}}} \leq
{\blacktriangle^{\!\star^{\SMALL{\prime\prime}}}}$. In particular,
$\sust \leq \susf \leq \sus$.

\begin{theorem} \label{blacktriangle-f}
Let $\star$ be a semistar operation of $D$ and let $\sus$ denote
the strict extension $\susK$ of $\star$ to $D[X]$ introduced in
{\rm Theorem \ref{strictext}}. Then:
\begin{enumerate}
\item $(\sus)_{\!{_{f}}}=(\susf)_{\!{_{f}}}$ is the
largest finite-type strict extension of $\stf$.

\item $\widetilde{\sus}=\widetilde{\blacktriangle^{\tilde{\star}}}$ is
the largest stable finite-type strict extension of $\tilde{\star}$.
\end{enumerate}
\end{theorem}
\begin{proof}
(1) Since $\sus$ is a strict extension of $\star$,
it follows immediately  that $(\sus)_{\!{_{f}}}$
is a strict extension of $\stf$. Then, by
Theorem \ref{strictext}(4), $(\sus)_{\!{_{f}}}\leq \susf$ and
hence $(\sus)_{\!{_{f}}}\leq (\susf)_{\!{_{f}}}$. Since the
opposite inequality is obvious, we have
$(\sus)_{\!{_{f}}}=(\susf)_{\!{_{f}}}$. Now, let $\bigstar$ be a finite-type
strict extension of $\stf$, then $\bigstar\leq
\blacktriangle^{\stf}$ by Theorem \ref{strictext}(4), and hence
$\bigstar=\bigstar_{\!{_{f}}}\leq
(\blacktriangle^{\stf})_{\!{_{f}}}=(\blacktriangle^{\star})_{\!{_{f}}}$.

(2) We will show first that $\widetilde{\sus}$ is a strict
extension of $\tilde{\star}$. Let $E\in \FF(D)$. Since $\sus$ is
an extension of $\star$, $\widetilde{\sus}$ is an extension of
$\tilde{\star}$ \cite[Lemma 5]{CF2011}, and hence we have
$E^{\tilde{\star}}[X]\subseteq (E[X])^{\widetilde{\sus}}$. Let $0
\neq f\in (E[X])^{\widetilde{\sus}}\subseteq
(E[X])^{\sus}=E^{\star}[X]\subseteq K[X]$. Then  $fJ\in E[X]$ for
some $J\in \f(D[X])$ such that $J\subseteq D[X]$ and
$J^{\sus}=(D[X])^{\sus}$. 
Since $J\subseteq c_D(J)[X]\subseteq D[X]$ and $J^{\sus}=(D[X])^{\sus}$,
we have $\co_D(J)^\star[X]=
(\co_D(J)[X])^{\sus}=(D[X])^{\sus}=D^\star[X]$, i.e.,
$(\co_D(J))^{\star}=D^{\star}$. Write $J=(g_1,g_2, \cdots , g_n)$.
Since $fg_i\in fJ\subseteq E[X]$, $\co_D(fg_i)\subseteq E$ for all
$i=1, 2, \cdots , n$. Let $m:=\text{deg}\, f$. Then, by Dedekind-Mertens Lemma
\cite[Theorem 28.1]{G},
$\co_D(f)\co_D(g_i)^{m+1}=\co_D(fg_i)\co_D(g_i)^{m}\linebreak \subseteq E$
for all $i=1, 2,  \cdots, n$, and so
$$\co_D(f)((\co_D(g_1))^{m+1}+
(\co_D(g_2))^{m+1}+ \cdots + (\co_D(g_n))^{m+1})\subseteq E.$$
Note that $(\co_D(g_1))^{m+1}+(\co_D(g_2))^{m+1}+ \cdots +
 (\co_D(g_n))^{m+1}$ is a finitely generated integral ideal of $D$.
Also, from the equation $(\co_D(g_1)+ \co_D(g_2)+\cdots +
\co_D(g_n))^{\star}=(\co_D(J))^{\star}=D^{\star}$, it easily
follows that
$$((\co_D(g_1))^{m+1} +(\co_D(g_2))^{m+1} + \cdots +
(\co_D(g_n))^{m+1})^{\star}=D^{\star}.$$ Therefore,
$\co_D(f)((\co_D(g_1))^{m+1}+ (\co_D(g_2))^{m+1} +\cdots +
(\co_D(g_n))^{m+1})\subseteq E$ implies that $\co_D(f)\in
E^{\tilde{\star}}$, i.e., $f\in E^{\tilde{\star}}[X]$. Thus, we
have $(E[X])^{\widetilde{\sus}}\subseteq E^{\tilde{\star}}[X]$ and
hence $E^{\tilde{\star}}[X] = (E[X])^{\widetilde{\sus}}$.
Therefore, $\widetilde{\sus}$ is a strict extension of
$\widetilde{\star}$ that is stable and of finite type.

By Theorem \ref{strictext}(4), $\widetilde{\sus}\leq
\blacktriangle^{\tilde{\star}}$ and hence $\widetilde{\sus}\leq
\widetilde{\blacktriangle^{\tilde{\star}}}$. Since the opposite
inequality is obvious, we have
$\widetilde{\sus}=\widetilde{\blacktriangle^{\tilde{\star}}}$. Let
$\bigstar$ be a stable finite-type strict extension of
$\tilde{\star}$. Then $\bigstar \leq
\blacktriangle^{\tilde{\star}}$ and hence
$\bigstar=\widetilde{\bigstar}\leq \widetilde{\
\blacktriangle^{\widetilde{\star}}\ }=\widetilde{\sus}$.
\end{proof}

\begin{corollary}\label{cor9}
Let $t_{D[X]}$ and $w_{D[X]}$ be the $t$-semistar operation and the
$w$-semistar operation of $D[X]$, respectively. Then:
\begin{enumerate}
\item $({\blacktriangle^{v_D}})_{\!{_{f}}}=({\blacktriangle^{t_D}})_{\!{_{f}}}=t_{D[X]}$.

\item $\widetilde{\blacktriangle^{v_D}}=\widetilde{\blacktriangle^{w_D}}=w_{D[X]}$.
\end{enumerate}
\end{corollary}
\begin{proof}
(1) Since $t_{D[X]}$ is the largest finite-type (semi)star
operation of $D[X]$ and it is a strict extension of $t_D$ (as
observed in  Section 1), it is the largest finite-type strict
extension of $t_D$ and hence, by Theorem~\ref{blacktriangle-f}(1),
$t_{D[X]}=({\blacktriangle^{t_D}})_{\!{_{f}}}=({\blacktriangle^{v_D}})_{\!{_{f}}}$.

(2) It follows from  Theorem~\ref{blacktriangle-f}(2)
and the fact that $\widetilde{\ t_{D[X]}}=w_{D[X]}$. Indeed,
$w_{D[X]}=\widetilde{\ t_{D[X]}}=\widetilde{\;\,
({\blacktriangle^{v_D}})_{\!{_{f}}}\, }=\widetilde{\,\
\blacktriangle^{v_D}\,} =\widetilde{\
\blacktriangle^{\widetilde{v_D}}}=\widetilde{\
\blacktriangle^{w_D}}$.
\end{proof}

\smallskip

It is natural to ask whether $(\susf)_{\!{_{f}}}= \susf$ and
$\widetilde{\blacktriangle^{\tilde{\star}}}={\blacktriangle^{\tilde{\star}}}$.
The next proposition provides the negative answer to that.

\begin{proposition} \label{prop10}
 With the notation of {\rm Theorem \ref{blacktriangle-f}}, the semistar operation $\blacktriangle^{\star}$
of $D[X]$ is not of finite type for any semistar operation $\star$
of $D$.
\end{proposition}
\begin{proof}
Let $A:= \bigcup_{n=1}^{\infty}\frac{1}{X^n}D[X]$, then $A \in
\boldsymbol{\overline{F}}(D[X])$ and $(K[X]:A) = (0)$. Hence
$A^{\blacktriangle^{\star}} = K(X)$ by definition. Next, if $B
\in\f(D[X])$ and $B \subseteq A$, then $B \subseteq
\frac{1}{X^m}D[X]$ for some  $m \geq 1$.
So $B^{\blacktriangle^{\star}} \subseteq
(\frac{1}{X^m}D[X])^{\blacktriangle^{\star}} =
\frac{1}{X^m}(D[X])^{\blacktriangle^{\star}} =
\frac{1}{X^m}D^{\star}[X] \subseteq K[X, \frac{1}{X}]$. Thus,
$A^{({\blacktriangle^{\star}})_{_{\!f}}}\subseteq K[X,
\frac{1}{X}] \subsetneq K(X) = A^{\blacktriangle^{\star}}$, which
implies $({\blacktriangle^{\star}})_{_{\!f}} \lneq
{\blacktriangle^{\star}}$.
\end{proof}

\smallskip

In the  following, we compare $\blacktriangle^{v_D}$,
$\blacktriangle^{v_D}_D$, and  $v_{D[X]}$.

\begin{corollary} \label{black-v}
Let $v_{D[X]}$ be the $v$-semistar operation of $D[X]$. Then:
\begin{enumerate}
\item $\blacktriangle^{\!v_D} \leq \blacktriangle^{\!v_D}_D = v_{D[X]}$.   Moreover,
$\blacktriangle^{\!v_D}=v_{D[X]}$ if and only if $D=K$.

\item $(v_{D[X]})_f=t_{D[X]} = (\blacktriangle^{\!v_D})_{\!{_f}} =
(\blacktriangle^{\!t_D})_{\!{_f}} \lneq \blacktriangle^{\!t_D}$.
\end{enumerate}
\end{corollary}
\begin{proof}
(1) The inequality  $\blacktriangle^{\!v_D}=\blacktriangle^{\!v_D}_K \leq
\blacktriangle^{\!v_D}_D$ holds by Theorem \ref{strictext}(3).
Also, since $\blacktriangle^{\!v_D}_D$ is a (semi)star operation of $D[X]$,
$\blacktriangle^{\!v_D}_D \leq v_{D[X]}$.
Now, for $A \in \F(D[X])$, we have
\begin{eqnarray*}
A^{ \blacktriangle^{\!v_D}_D}\hskip -8pt & \supseteq &  \hskip -4pt
 \bigcap\{z^{-1}(\co_D (z A))^{v_D}[X] \mid 0 \neq z \in (D[X]:A)\} \\
 \hskip -8pt &= &  \hskip -4pt   \bigcap\{z^{-1}(\co_D (z A)[X])^{v_{D[X]}} \mid 0 \neq z\in (D[X]:A)\} \\
\hskip -8pt & \supseteq &  \hskip -4pt  \bigcap\{z^{-1}(z A)^{v_{D[X]}}  \mid 0 \neq z\in (D[X]:A)\} \\
\hskip -8pt &= &  \hskip -4pt  A^{v_{D[X]}},
\end{eqnarray*}
and for $A\in \FF(D[X])\setminus \F(D[X])$, we have
$A^{v_{D[X]}}=K(X)=A^{ \blacktriangle^{\!v_D}_D}$. Thus we obtain
the equality $\blacktriangle^{\!v_D}_D = v_{D[X]}$.
If $D=K$, then obviously $\blacktriangle^{\!v_D}=\blacktriangle^{\!v_D}_K=
\blacktriangle^{\!v_D}_D = v_{D[X]}$. Assume that $D\neq K$. Then $K\in \FF(D)\setminus \F(D)$
and for each $E\in \FF(D)\setminus \F(D)$, $E[X]^{v_{D[X]}}=K(X)$ but $E^{v_D}[X]=K[X]$.
Thus $\blacktriangle^{\!v_D}_D = v_{D[X]}$
is not a strict extension of $v_D$ and hence $\blacktriangle^{\!v_D} \lneq \blacktriangle^{\!v_D}_D =
v_{D[X]}$.

(2) is an easy consequence of Corollary~\ref{cor9}(1)
 and Proposition~\ref{prop10}.
\end{proof}

\begin{remark}\label{rem-ex} \rm
(a)\, The statement (1) of the previous corollary can be
stated more precisely as follows: for each $A \in \FF(D[X])$, we have
 $$
A^{v_{D[X]}}= \left\{
\begin{array}{ll}
K(X) =A^{\blacktriangle^{\!v_D}_D} = A^{\blacktriangle^{\!v_D}} \,
& \mbox{ if }
(K[X] : A) =(0)\,, \\
K(X) = A^{\blacktriangle^{\!v_D}_D} \supsetneq
A^{\blacktriangle^{\!v_D}}  \,&  \mbox{ if }
(K[X] :A)\neq (0) \mbox{ but }  (D[X] :A) =(0)\,,\\
 A^{\blacktriangle^{\!v_D}_D} = A^{\blacktriangle^{\!v_D}}  \,&  \mbox{ if }
(D[X] :A)\neq (0). 
\end{array}
\right.
$$
%

(b)\, If $D$ is a Krull domain, then $D[X]$ is also a Krull
domain, and hence 
$t_{D[X]}$, ${\blacktriangle^{\!v_D}}$,
${\blacktriangle^{\!v_D}_D}$, and $ v_{D[X]}$ coincide when they
are restricted to $\F(D[X])$. Therefore,   the star operation
${\blacktriangle^{\!v_D}}|_{\F(D[X])}$ can be of finite type, in
contrast with the semistar operation ${\blacktriangle^{\!v_D}}$
(Proposition \ref{prop10}). On the other hand,
 if $D$ is a TV-domain such that $D[X]$ is not a TV-domain
(see \cite{CR}), then by (a), ${\blacktriangle^{t_D}}|_{\F(D[X])}
= {\blacktriangle^{v_D}}|_{\F(D[X])}=v_{D[X]}|_{\F(D[X])}\neq
t_{D[X]}|_{\F(D[X])}$. Thus ${\blacktriangle^{t_D}}|_{\F(D[X])}$
(and ${\blacktriangle^{t_D}}$) may not be of finite type.

(c)\,  Note that ${\blacktriangle}^{e_D} \lneq e_{D[X]}$
and $d_{D[X]}\leq ({\blacktriangle}^{d_D})_f$; moreover,
$d_{D[X]}=({\blacktriangle}^{d_D})_f$ if and only if $D=K$. The
first proper inequality is obvious, because
${\blacktriangle}^{e_D} \leq e_{D[X]} $  \ and
$(D[X])^{{\blacktriangle}^{e_D}} $ $= D^{e_D}[X] =K[X] \subsetneq
K(X) = (D[X])^{e_{D[X]}}$. The second inequality is also obvious,
because $d_{D[X]}$ is the smallest semistar operation of $D[X]$.
If $D=K$, then $({\blacktriangle}^{d_D})_f=
({\blacktriangle}^{v_D})_f=t_{D[X]}=d_{D[X]}$ by
Corollary~\ref{black-v}. Assume that $D\neq K$.  Let $\alpha$ be
a nonzero nonunit element of $D$ and let $A=(\alpha, X)D[X]$.
Since $(K[X]:A)=K[X]$ and $\co_D(zA)=\co_D(z)$ for all $z\in
K[X]\setminus (0)$, $ A=A^{d_{D[X]}} \subsetneq
A^{(\blacktriangle^{d_D})_{\!{_f}} }
=A^{\blacktriangle^{d_D}}=D[X]  $. In fact, since $1 \in
(K[X]:A)$, $A^{\blacktriangle^{d_D}}\subseteq D[X]$; on the other
hand, $A^{\blacktriangle^{d_D}}= \bigcap \{z^{-1} \co_D(zA)[X]
\mid 0 \neq z \in K[X]\} = \bigcap \{z^{-1} \co_D(z)[X] \mid 0
\neq z \in K[X]\} \supseteq  \bigcap \{z^{-1} zD[X]\mid 0 \neq z
\in K[X]\} =D[X]$.
%
\end{remark}

\smallskip

According to Theorem \ref{strictext}, any semistar operation $\star$ of $D$ admits a strict extension (in
fact, the largest strict extension) to $D[X]$. In the following, we show that, in fact, there
exist infinitely many strict extensions to $D[X]$ of $\star$.



%

\begin{lemma}\label{strict-ext} Let $\star$ be a semistar operation  of $D$
and let $\bigstar'$ and $\bigstar''$ be   two extensions of
$\star$ to $D[X]$. If $\bigstar' \leq \bigstar''$ and $\bigstar''$
is a strict extension of $\star$, then $\bigstar'$ is also a
strict extension of $\star$.
\end{lemma}
\begin{proof}  For each $E\in \FF(D)$, we have $E^\star[X]  \subseteq (E[X])^{\bigstar'}  \subseteq
(E[X])^{\bigstar''} = E^\star[X]$, and hence $(E[X])^{\bigstar'} =
E^\star[X]$.
\end{proof}

 \begin{proposition}
For each semistar operation $\star$ of $D$, 
there exists a strictly increasing infinite sequence of semistar
operations of $D[X]$ which are all strict extensions of $\star$.
\end{proposition}
\begin{proof} Let $\{f_i\}_{i=1}^\infty$ be a set of countably
infinite nonassociate irreducible polynomials in $K[X]$.  For each
$n\geq 1$, let $\bigstar_n:=(\bigwedge
\{\bigstar_{\{K[X]_{(f_i)}\}}\mid i\geq n\}) \wedge
\blacktriangle^\star$, i.e., for each $A\in \FF(D[X])$,
$A^{\bigstar_n}:=(\bigcap_{i\geq n} AK[X]_{(f_i)}) \cap
A^{\blacktriangle^\star}$. Then each $\bigstar_n$ is a strict
extension to $D[X]$ of $\star$  such that $\bigstar_1\leq
\bigstar_2\leq\cdots \leq \blacktriangle^\star$ (Lemma
\ref{strict-ext} and Theorem \ref{strictext}(4)). Let $n<m$ and
let $B:=K[X]_{(f_n)}$. Then $B\in \FF(D[X])$ and
$BK[X]_{(f_i)}=K(X)$ for all $i\neq n$. Since $(K[X]: B)=(0)$,
$B^{\blacktriangle^\star}=K(X)$. Therefore, $B^{\bigstar_m}=K(X)$,
while $B^{\bigstar_n}=B$. Thus $\bigstar_n\neq \bigstar_m$.
\end{proof}

\begin{remark} \rm
We do not know whether, for a given arbitrary {\sl star} operation
$\ast$ of $D$, there exists a strictly increa\-sing sequence of
(strict) {\sl star} operation extensions to $D[X]$ of $\ast$.
However, we can show that it does hold if $\ast$ is of finite type
and $D$ admits a $\ast$-valuation overring which is not equal to
the quotient field $K$.

Let $V$ be a nontrivial $\ast$-valuation overring of $D$ with
maximal ideal $M$ and let $P:=M\cap D$. Then $P$ is a nonzero
prime ideal of $D$.

{\bf Case 1.}  {\sl $D/P$ is infinite}.

Let $\{a_i\}_{i=1}^\infty$ be a set of elements of $D$ such that
$\bar a_i\neq \bar a_j$ in $D/P$ for all $i\neq j$. Then the
ideals $N_i:=M+(X-a_i)V[X]$  are distinct maximal ideals of
$V[X]$. We denote by $\blacktriangle^*_D$ the strict star
operation extension to $D[X]$ of $\ast$ (defined as in the
semistar operation case, but obviously only on nonzero fractional
ideals of $D[X]$). For each $n\geq 1$, define
$A^{{\bast}_n}:=(\bigcap_{i\geq n} AV[X]_{N_i})\cap
A^{\blacktriangle^\ast_D}$ for all $A\in \F(D[X])$. Since $V$ is a
$\ast$-valuation overring of $D$ and $\blacktriangle^\ast_D$ is a
strict star operation extension to $D[X]$ of $\ast$, each
${\Bast}_n$ is a strict star operation extension to $D[X]$ of
$\ast$ such that ${\Bast}_1\leq {\Bast}_2\leq\cdots$. Choose a
nonzero element $c\in P$ and let $A_i:=(c, X-a_i)D[X]$ for each $i
\geq 1$. Then, since $\co_D(X-a_i)=D$, we have
$A_i^{\blacktriangle^\ast_D}=D[X]$. Let $n<m$. Then
$A_n^{\bast_n}\subseteq N_nV[X]_{N_n}\cap D[X]\subseteq N_n[X]\cap
D[X]\subsetneq D[X]$, while $A_n^{\bast_m}=(\bigcap_{i\geq m}
V[X]_{N_i})\cap D[X]=D[X]$. Thus ${\Bast}_n\neq {\Bast}_m$.


{\bf Case 2.}  {\sl $D/P$ is finite.}

Since $D/P$ is a finite field, for each $n\geq 1$, there exists a
monic irreducible polynomial $\bar f_n\in (D/P)[X]$ of degree $n$.
Then, for $n\neq m$, $(\bar f_n, \bar f_m)(D/P)[X]=(D/P)[X]$, and
so, {\sl a fortiori}, $(\bar f_n, \bar f_m)(V/M)[X]=(V/M)[X]$.
Since $\bar f_n$ is a nonconstant polynomial (in $(D/P)[X]$) and
hence, in particular,  a nonunit element in $(V/M)[X]$, there
exists a maximal ideal $\bar N_n$ of $(V/M)[X]$ containing $\bar
f_n$. Note that, by the previous observations, $\bar N_n\neq \bar
N_m$ for $n\neq m$.  Let $\varphi: V[X]\rightarrow (V/M)[X]$ be
the canonical epimorphism, set $N_n:=\varphi^{-1}(\bar N_n)$, and
let $f_n$ be a monic polynomial in $D[X]$ of degree $n$ such that
$\varphi(f_n)=\bar f_n$. For each $n\geq 1$, define
$A^{\bast_n}:=(\bigcap_{i\geq n} AV[X]_{N_i})\cap
A^{\blacktriangle^*_D}$ for all $A\in \F(D[X])$. Then, as above,
each $\Bast_n$ is a strict star operation extension to $D[X]$ of
$\ast$ such that $\Bast_1\leq \Bast_2\leq\cdots$. Choose a nonzero
element $c\in P$ and set $A_n:=(c, f_n)D[X]$. Then, since
$\co_D(f_n)=D$, we have $A_n^{\blacktriangle^\ast_D}=D[X]$. Let
$n<m$. Then $A_n^{\bast_n}\subseteq N_nV[X]_{N_n}\cap
D[X]\subseteq N_n[X]\cap D[X]\subsetneq D[X]$, while
$A_n^{\bast_m}=(\bigcap_{i\geq m} V[X]_{N_i})\cap D[X]=D[X]$.
Thus $\Bast_n\neq \Bast_m$.
%
\end{remark}
{\section{Relationship among strict extensions} }


In Section 2, we have shown that each semistar operation of
$D$ admits the largest strict extension to $D[X]$, by defining its
precise form. The next proposition provides the existence of the
smallest strict extension. Unlike the largest strict
extension, the smallest strict extension  is not described in an
explicit form in general.  However, 
for a stable semistar operation of finite type, we are able to
provide a complete description of its  smallest strict extension.

\smallskip

\begin{proposition} \label{curly}
Let $D$ be an integral domain and let $\star$ be a semistar
operation of $D$.
Set $\{\star_\lambda \mid \lambda \in \Lambda \}$ the set of all
the semistar operations of $D[X]$ extending $\star$. Then
$\boldsymbol{\curlywedge}^\star := \wedge \{ \star_{\lambda} \mid
\lambda \in \Lambda \}$ is the smallest semistar operation of
$D[X]$ extending $\star$.
Moreover, it is a strict extension of $\star$.
\end{proposition}
\begin{proof}  Note that, by definition, for all $E \in \FF(D)$,
$(E[X])^{\boldsymbol{\curlywedge}^\star} = \bigcap \{(E[X])^{\star_\lambda} \mid \lambda \in \Lambda\}$.
Since $(E[X])^{\star_\lambda} \cap K = E^\star$ for each $\lambda \in \Lambda$,
we deduce immediately that $(E[X])^{\boldsymbol{\curlywedge}^\star} \cap K = E^\star$.
Also, $\sus$ is a strict extension of $\star$ to $D[X]$ by Theorem \ref{strictext}(4),
and so $\boldsymbol{\curlywedge}^\star \leq \sus$. Therefore,
$\boldsymbol{\curlywedge}^\star$ is a strict extension of $\star$
by Lemma \ref{strict-ext}.
\end{proof}

\smallskip

We know that for any semistar operation $\star$ of
$D$, $(\sus)_{\!{_{f}}}=(\susf)_{\!{_{f}}}\lneq \susf$ and
$\widetilde{\blacktriangle^{\star}}=\widetilde{\left(\sust \right)} \lneq \sust$. We will
see now what happens for the semistar operation
$\boldsymbol{\curlywedge}^\star $.

\begin{lemma}\label{leq}  Let $\star$, $\star'$ and $\star''$ be
semistar operations of $D$ 
and let $\boldsymbol{\curlywedge}^\star $ be the smallest strict extension of $\star$ to $D[X]$
introduced in {\rm Proposition \ref{curly}}. Then:
\begin{enumerate}
\item $\boldsymbol{\curlywedge}^\star =\wedge \{
\bigstar \mid \bigstar \mbox{ is a semistar operation of } D[X] \mbox{
such that }  \bigstar_0  \geq \star \}$.

\item If  $\star' \leq \star''$,  then
$\boldsymbol{\curlywedge}^{\star'} \leq
\boldsymbol{\curlywedge}^{\star''}$.

\item Every semistar operation $\bigstar$ of $D[X]$
such that $\boldsymbol{\curlywedge}^\star \leq \bigstar \leq \sus$
is a strict extension of $\star$.
\end{enumerate}
\end{lemma}
\begin{proof} (1) Set $ \blacklozenge := \wedge \{ \bigstar
\mid \bigstar \mbox{ is a semistar operation of } D[X] \mbox{ such that
} \bigstar_0  \geq \star \}$. It is obvious that $\blacklozenge
\leq \boldsymbol{\curlywedge}^\star. $ From this inequality, we
obtain that  $\star \leq \blacklozenge_0 \leq
(\boldsymbol{\curlywedge}^\star)_0 = \star$ and therefore
$\blacklozenge$ is an extension of $\star$ to $D[X]$. By the
minimality of $\boldsymbol{\curlywedge}^\star$, we deduce that
$\blacklozenge$ coincides with $\boldsymbol{\curlywedge}^\star$.

(2) is a straightforward consequence of (1).

(3) is an easy consequence of Lemma \ref{strict-ext}.
\end{proof}

\begin{proposition}\label{curlyft}
Let $\star$ be a semistar operation of $D$
and let $\boldsymbol{\curlywedge}^\star $ be the semistar
operation of $D[X]$ introduced in {\rm Proposition \ref{curly}}. Then:
\begin{enumerate}
\item $\boldsymbol{\curlywedge}^{\stf}$ is a semistar operation
of finite type {\rm(}and hence $\boldsymbol{\curlywedge}^{\stf}
=(\boldsymbol{\curlywedge}^{\stf})_{\!{_{f}}}\leq
(\boldsymbol{\curlywedge}^{\star})_{\!{_{f}}} \leq
(\sus)_{\!{_{f}}}=(\susf)_{\!{_{f}}}\lneq \susf${\rm)}.

\item $\boldsymbol{\curlywedge}^{\stt}$ is a stable
semistar operation of finite type  {\rm(}and hence
$\boldsymbol{\curlywedge}^{\stt}=\widetilde{\ \boldsymbol{\curlywedge}^{\stt}} \leq
\widetilde{\boldsymbol{\curlywedge}^{\star}} \leq
\widetilde{\blacktriangle^{\star}}=\widetilde{\left(\sust \right)}
 \lneq \sust${\rm)}.
\end{enumerate}
\end{proposition}
\begin{proof}
(1) By \cite[Lemma 5]{CF2011}, we have $\stf= (\stf)_{\!{_{f}}}=
((\boldsymbol{\curlywedge}^{\stf})_0)_{\!{_{f}}} =
((\boldsymbol{\curlywedge}^{\stf})_{\!{_{f}}})_0\leq
(\boldsymbol{\curlywedge}^{\stf})_0 = \stf$, and hence
$(\boldsymbol{\curlywedge}^{\stf})_{\!{_{f}}} \leq
\boldsymbol{\curlywedge}^{\stf}$ are both extensions of $\stf$.
By the minimality of $\boldsymbol{\curlywedge}^{\stf}$, they
must be equal. Thus $\boldsymbol{\curlywedge}^{\stf}$ is a strict
extension of $\stf$ of finite type. The parenthetical statement is
a straightforward consequence of  Theorem \ref{blacktriangle-f}(1)
and Proposition \ref{curly}.

(2) By \cite[Lemma 5]{CF2011}, $\left(\widetilde{\ \boldsymbol{\curlywedge}^{\stt}}\right)_{\! 0}
= \widetilde  {\ \! ({\boldsymbol{\curlywedge}^{\stt}})_{0} }=
\widetilde{\stt} =\stt
$, and hence $\widetilde{\ \boldsymbol{\curlywedge}^{\stt}} \leq
{\boldsymbol{\curlywedge}^{\stt}}$ are both extensions of $\stt$.
 Then, by the minimality of ${\boldsymbol{\curlywedge}^{\stt}}$,
we have $\widetilde{\ \boldsymbol{\curlywedge}^{\stt}} =
{\boldsymbol{\curlywedge}^{\stt}} $. Thus ${\boldsymbol{\curlywedge}^{\stt}} $
is stable and of finite type. The parenthetical statement is a
straightforward consequence of  Theorem \ref{blacktriangle-f}(2)
and Proposition \ref{curly}.
\end{proof}

\begin{remark} \label{remark4.8} \rm
(a)\, It can happen that
$(\boldsymbol{\curlywedge}^{\star})_{\!{_{f}}} \lneq
(\sus)_{\!{_{f}}}$. For instance, if $D$ is not a field, then
$\boldsymbol{\curlywedge}^{d_D}=d_{D[X]}\lneq
(\blacktriangle^{d_D})_{\!{_{f}}}$ (see Remark~\ref{rem-ex}(c)).
But, at the moment, we do not know if it is possible that
$\boldsymbol{\curlywedge}^{\stf} \lneq
(\boldsymbol{\curlywedge}^{\star})_{\!{_{f}}}$.

(b)\, It can happen that $\widetilde{\boldsymbol{\curlywedge}^{\star}}
\lneq \widetilde{\blacktriangle^{\star}}$. For instance, let $D$
be an integral domain, not a field, with $d_D = w_D$. Then
$\boldsymbol{\curlywedge}^{w_D} = d_{D[X]}$,
$\widetilde{(\blacktriangle^{w_D})} = w_{D[X]}$ by
Corollary~\ref{cor9}, but $d_{D[X]} \neq w_{D[X]}$. So
$\boldsymbol{\curlywedge}^{w_D}=\widetilde{\boldsymbol{\curlywedge}^{w_D}}
\lneq \widetilde{\blacktriangle^{w_D}}$. On the other hand, we do
not know whether  it  is possible that
$\boldsymbol{\curlywedge}^{\stt}\lneq
\widetilde{\boldsymbol{\curlywedge}^{\star}}$.
\end{remark}

\smallskip

Let $\star$ be a semistar operation of $D$. 
In the paper \cite{CF2011}, the authors introduced the following semistar operations 
$[\stt]$ and $\langle{\stt}\rangle$ of $D[X]$: for each $A \in \FF(D[X])$,
$$
\begin{array}{rl}
A^{[\stt]} := & \hskip -6 pt \bigcap \{ AD_Q[X]\mid Q \in \QMax^{\stf}(D)\}\,,\\
A^{\langle{\stt}\rangle} :=& \hskip -6pt \bigcap \{ AD_Q(X) \mid Q
\in\QMax^{\stf} \} \cap AK[X].
\end{array}
$$
They showed that both are stable finite-type strict extensions of $\stt$ \cite[Corollary 18]{CF2011}.

We will compare  $\boldsymbol{\curlywedge}^{\stt}$, $\widetilde{\blacktriangle^{\star}}$,
$[\stt]$, and $\langle{\stt}\rangle$.

\begin{theorem} \label{ok-stable}
Let $\star$ be a semistar operation of $D$.
Then $\boldsymbol{\curlywedge}^{\stt}=[\stt]$, i.e.,
$$A^{\boldsymbol{\curlywedge}^{\stt}}=\bigcap\{ AD_Q[X]\mid Q \in \QMax^{\stf}(D)\}$$ for each $A \in \FF(D[X])$.
\end{theorem}
\begin{proof}
Since $[\stt]$  is an extension of $\stt$ to $D[X]$, it suffices to show that
$[\stt]\leq \boldsymbol{\curlywedge}^{\stt}$.
Note that 
$\boldsymbol{\curlywedge}^{\stt}$ and $[\stt]$ are
stable semistar operations of finite type, so
$\widetilde{\boldsymbol{\curlywedge}^{\stt}}=\boldsymbol{\curlywedge}^{\stt}$ and $\widetilde{[\stt]}=[\stt]$.
Therefore, for each $A \in \FF(D[X])$, we have:
$$
\begin{array}{rl}
A^{\boldsymbol{\curlywedge}^{\stt}}
 =&\bigcup \{(A:J) \mid J \in \f(D[X]),\ J \subseteq D[X], \mbox{ and }
J^{\boldsymbol{\curlywedge}^{\stt}}=(D[X])^{\boldsymbol{\curlywedge}^{\stt}}
=D^{\stt}[X]\} \\
\supseteq & \bigcup \{(A:H[X]) \mid H \in \f(D),\ H \subseteq D, \mbox{ and }
(H[X])^{\boldsymbol{\curlywedge}^{\stt}} =D^{\stt}[X]\}\\
= & \bigcup \{(A:H[X]) \mid H \in \f(D),\ H \subseteq D, \mbox{ and } H^{\stt}[X]
=D^{\stt}[X]\} \\
= & \bigcup \{(A:H) \mid H \in \f(D),\ H \subseteq D, \mbox{ and }
H^{\stt} =D^{\stt}\}
\\
= & A^{[\stt]}.
\end{array}$$
\noindent Thus the conclusion $\boldsymbol{\curlywedge}^{\stt}=[\stt]$ follows.
\end{proof}

\smallskip

Finally we will show that $\widetilde{\blacktriangle^{\star}}=\langle{\stt}\rangle$.
For this purpose, we need to extend a couple of results which are
known for the $t$-operation case to a more general semistar operation
setting.

Given a semistar operation $\star$ of $D$, let $v_D(D^\star)$
be the semistar operation of $D$ defined by $E^{v_D(D^\star)}
:=(ED^\star)^{v_{D^\star}}=(D^\star : (D^\star : E))$ for each $E
\in \FF(D)$ and set ${t_D(D^\star)} := {v_D(D^\star)_{\!{_f}}}$.
Then $E^{t_D(D^\star)}=(ED^\star)^{t_{D^\star}}$ for each $E \in
\FF(D)$. It is also obvious that $\star \leq v_D(D^\star)$ and $\stf
\leq t_D(D^\star)$.

\begin{lemma}\label{extended}
Let $\star$ be a semistar operation of $D$
and let $t(D^\star[X]):=t_{D[X]}(D^\star[X])$ be the semistar
operation of $D[X]$ introduced above. If $M$ is a
quasi-$t(D^\star[X])$-maximal ideal of $D[X]$ with $M\cap D \neq
(0)$, then $M = (M\cap D)[X]$.
\end{lemma}
\begin{proof}
Let $M$ be a quasi-$t(D^\star[X])$-maximal ideal of
$D[X]$ with $M\cap D\neq (0)$. Since
$M^{t(D^\star[X])}=(MD^\star[X])^{t_{D^\star[X]}}$ is a proper
ideal of $D^\star[X]$, there exists a $t_{D^\star[X]}$-maximal
ideal $N$ of $D^\star[X]$ containing
$(MD^\star[X])^{t_{D^\star[X]}}$. Since $N\cap D^\star\neq (0)$,
$N=(N\cap D^\star)[X]$ \cite[Proposition 1.1]{HZ2}. Therefore, it follows that $M=N\cap
D[X]=(N\cap D^\star)[X]\cap D[X]=(M\cap D)[X]$.
\end{proof}

\begin{lemma} \label{upper}
Let $\star$ be a semistar operation of $D$
and let $t(D^\star) :=t_D(D^\star)$, \linebreak $t(D^\star[X])
:=t_{D[X]}(D^\star[X])$ be as above.
If $Q$ is a nonzero prime ideal of $D[X]$ such that $Q
\cap D = (0)$ and $\co_D(Q)^{t(D^\star)} = D^\star$, then $Q$ is a
quasi-$t(D^\star[X])$-maximal ideal of $D[X]$.
\end{lemma}
\begin{proof} It is clear that $Q$ is a quasi-$t(D^\star[X])$-prime ideal of
$D[X]$. Suppose $Q$ is not a quasi-$t(D^\star[X])$-maximal ideal
of $D[X]$ and let $M$ be a quasi-$t(D^\star[X])$-maximal ideal of
$D[X]$ with $Q \subseteq M$. Since the containment is proper, $M
\cap D \neq (0)$. By Lemma \ref{extended}, $M= (M\cap D)[X]$. Note
that $M=M^{t(D^\star[X])}=(M\cap D)^{t(D^\star)}[X]$ and hence
that $(M\cap D)^{t(D^\star)}=M\cap D$. Since $Q \subseteq M$, we
have $\co_D(Q) \subseteq \co_D(M)=M\cap D$ and so
$\co_D(Q)^{t(D^\star)} \subseteq  (M\cap D)^{t(D^\star)}
\subsetneq D^\star$, which is a contradiction.
\end{proof}

\begin{proposition} \label{qmax-black}
Let $\star$ be a semistar operation of $D$.
Then
$$
\begin{array}{rl}
\QMax^{({\blacktriangle^{\stf}})_{f}}(D[X])  & \hskip -5pt=
 \{Q \in \Spec
(D[X]) \mid  Q \cap D = (0)  \mbox{ and } \co_D(Q)^{\stf} =
D^{\stf}\}
\\
 & \hskip 5pt  \bigcup \,\{P[X] \mid P \in \QMax^{\stf}(D)\}\,.
\end{array}
$$
\end{proposition}
\begin{proof}
Let $t(D^\star)$ and $t(D^\star[X])$ be as in Lemma \ref{upper}.
Then $\stf \leq t(D^\star)$ and $({\blacktriangle^{\stf}})_f \leq
t(D[X]^{\blacktriangle^{\stf}})=t(D^\star[X])$. Let $Q$ be a prime
ideal of $D[X]$ with $Q\cap D =(0)$ and $\co_D(Q)^{\stf} =
D^{\stf}$. Then, obviously, $\co_D(Q)^{t(D^\star)} = D^\star$, and
thus, by Lemma \ref{upper}, $Q$ is a quasi-$t(D^\star[X])$-maximal
ideal of $D[X]$. This implies that $Q$ is a
quasi-$({\blacktriangle^{\stf}})_f$-prime ideal of $D[X]$. Let $P
\in \QMax^{\stf}(D)$. Since $({\blacktriangle^{\stf}})_f$ is a
strict extension of $\stf$,
$(P[X])^{({\blacktriangle^{\stf}})_f}=P^{\stf}[X]$ and hence
$(P[X])^{({\blacktriangle^{\stf}})_f} \cap D[X]=P^{\stf}[X]\cap
D[X]=P[X]$. This implies that $P[X]$ is a
quasi-$({\blacktriangle^{\stf}})_f$-prime ideal of $D[X]$.

Therefore, for the equality of the statement, it suffices to show
that if $M$ is a prime ideal of $D[X]$ such that $M \cap D \neq
(0)$ and $\co_D(M)^{\stf} = D^{\stf}$, then
$M^{({\blacktriangle^{\stf}})_f} =
(D[X])^{({\blacktriangle^{\stf}})_f}$. Choose a nonzero $a \in M
\cap D$ and  a nonzero $g \in M$ with $\co_D(g)^{\stf} =
D^{\stf}$. Then, for each nonzero element $z\in (K[X]: (a,
g))\subseteq K[X]$, $\co_D(zg)^{\stf}=\co_D(z)^{\stf}$ by
Dedekind-Mertens Lemma, and so
$\co_D(z(a,g))^{\stf} = (\co_D(za)+\co_D(zg))^{\stf} =
(\co_D(za)^{\stf} + \co_D(zg)^{\stf})^{\stf} = (a\co_D(z)^{\stf} +
\co_D(z)^{\stf})^{\stf} = \co_D(z)^{\stf}$. Therefore, we have
$$
\begin{array}{rl}
(D[X])^{({\blacktriangle^{\stf}})_f} \supseteq
(a,g)^{{\blacktriangle^{\stf}}}
= &\hskip -5pt \bigcap\{z^{-1}(\co_D(z(a,g))^{\stf}[X] \mid 0\neq z \in (K[X] : (a, g))\}\\
= &\hskip -5pt\bigcap\{z^{-1}(\co_D(z))^{\stf}[X] \mid 0\neq z \in (K[X]:(a,g))\} \\
= &\hskip -5pt\bigcap \{z^{-1}(\co_D(z)[X])^{({\blacktriangle^{\stf}})_f} \mid 0\neq z \in (K[X]:(a,g))\} \\
\supseteq &\hskip -5pt \bigcap\{z^{-1}(zD[X])^{({\blacktriangle^{\stf})_f}} \mid  0\neq z \in (K[X]:(a,g))\} \\
= &\hskip -5pt (D[X])^{({\blacktriangle^{\stf}})_f}\,.
\end{array}
$$
Thus  $(a,g)^{({\blacktriangle^{\stf}})_f} =
(D[X])^{({\blacktriangle^{\stf}})_f}$, and hence
$M^{({\blacktriangle^{\stf}})_f} =
(D[X])^{({\blacktriangle^{\stf}})_f}$.
\end{proof}

\begin{theorem}
Let $D$ be an integral domain with quotient field $K$
and let $\star$ be a semistar operation of $D$.
Then $\widetilde{\blacktriangle^{\star}}={\langle{\stt}\rangle}$, i.e.,
$$A^{\widetilde{\blacktriangle^{\star}}}= \bigcap \{ AD_Q(X) \mid Q
\in\QMax^{\stf} \} \cap AK[X]$$ for each $A\in \FF(D[X])$.
\end{theorem}
\begin{proof} By Proposition \ref{qmax-black} and \cite[Remark
20(2)]{CF2011},
we have $\QMax^{({\blacktriangle^{\stf}})_{f}}(D[X]) =
\QMax^{{\langle{\stt}\rangle}}(D[X])$ and hence
$\widetilde{(\blacktriangle^{\stf})_{\!{_f}}} =
\widetilde{\,{\langle{\stt}\rangle}\,}$. By Theorem
\ref{blacktriangle-f},
$\widetilde{(\blacktriangle^{\stf})_{\!{_f}}}
=\widetilde{(\blacktriangle^{\star})_{\!{_f}}}
=\widetilde{\blacktriangle^{\star}}$, and by \cite[Proposition
16]{CF2011},
$\widetilde{\,{\langle{\stt}\rangle}\,}={\langle{\stt}\rangle}$.
Thus the conclusion $\widetilde{\blacktriangle^{\star}}
= {\langle{\stt}\rangle}$ follows.
\end{proof}


\section{Semistar Operations Defined by Families of Overrings}


In the present section, we generalize some known results
concerning the polynomial extensions of a stable finite-type
semistar operation, to the case where the semistar operation is
defined by a given family of overrings of $D$.

\begin{lemma} \label{lemma22}
Let $T$ be an overring of $D$ and let $\star_{\{T\}}$
{\rm(}respectively, $\bigstar_{\{T[X]\}}${\rm)} be the semistar operation of
$D$ {\rm(}respectively, of $D[X]${\rm)} defined by $E^{\star_{\{T\}}} := ET$
for each $E \in \FF(D)$ {\rm(}respectively, $A^{\bigstar_{\{T[X]\}}} :=
A T[X]$ for each $A \in \FF(D[X])${\rm)}. Then
$\boldsymbol{\curlywedge}^{\star_{\{T\}}} = {\bigstar_{\{T[X]\}}}
\lneq \blacktriangle^{\star_{\{T\}}} $.
\end{lemma}
\begin{proof} It is clear that
$\boldsymbol{\curlywedge}^{\star_{\{T\}}}$,
${\bigstar_{\{T[X]\}}}$, and $\blacktriangle^{\star_{\{T\}}} $ are
all strict extensions of $\star_{\{T\}}$.
Since $\bigstar_{\{T[X]\}}$ is of finite type but
$\blacktriangle^{\star_{\{T\}}}$ is not of finite type
(Proposition \ref{prop10}), we have
$\boldsymbol{\curlywedge}^{\star_{\{T\}}} \leq
{\bigstar_{\{T[X]\}}} \lneq \blacktriangle^{\star_{\{T\}}}$. For
the equality $\boldsymbol{\curlywedge}^{\star_{\{T\}}} =
{\bigstar_{\{T[X]\}}}$, let $A \in \FF(D[X])$. Then
$A^{\boldsymbol{\curlywedge}^{\!\star_{\{T\}}}}=
(AD[X])^{\boldsymbol{\curlywedge}^{\!\star_{\{T\}}}}\supseteq
A(D[X])^{\boldsymbol{\curlywedge}^{\!\star_{\{T\}}}}=
AD^{\star_{\{T\}}}[X]=AT[X]=A^{{\bigstar_{\{T[X]\}}}}$. This
implies $\boldsymbol{\curlywedge}^{\star_{\{T\}}} \geq
{\bigstar_{\{T[X]\}}}$. Therefore, the equality
$\boldsymbol{\curlywedge}^{\star_{\{T\}}} = {\bigstar_{\{T[X]\}}}$
holds.
\end{proof}

\smallskip

Now we consider a semistar operation given by 
an arbitrary family of overrings.  Let $\TT := \{ T_{\lambda} \mid
\lambda \in \Lambda \}$ be a set of overrings of $D$, and let
$\wedge_{\TT} := \bigwedge \{ {\star_{\{T_{\lambda}\}}} \mid
\lambda \in \Lambda \}$,
 i.e., $E^{\wedge_{\TT}}: =
\bigcap_{\lambda}ET_{\lambda}$ for each $E \in \FF(D)$.
Let $\TT[X] := \{ T_{\lambda}[X]  \mid \lambda \in \Lambda \}$ and let
$\bigwedge_{\TT[X]} := \bigwedge \{
{\bigstar_{\{T_{\lambda}[X]\}}}
\mid \lambda \in \Lambda \}$,  i.e., $A^{\bigwedge_{\TT[X]}}: =
\bigcap_{\lambda}AT_{\lambda}[X]$ for each $A \in \FF(D[X])$.

\begin{proposition}\label{prop23}
With the notation recalled above,  we have
\begin{center}
$\boldsymbol{\curlywedge}^{\wedge_{\TT}}  \leq  {{\bigwedge}}_{\TT[X]}
\leq \blacktriangle^{\wedge_{\TT}}\,. $
\end{center}
\end{proposition}
\begin{proof}
We easily deduce from the definitions that, for each $E \in
\FF(D)$,  $$ (E^{\star_{\TT}}[X])^{\bigwedge_{\TT[X]}} =
(E[X])^{\bigwedge_{\TT[X]}} = E^{\wedge_{\TT}}[X]\,,$$ and hence
${\bigwedge_{\TT[X]}}$ is a strict extension of $\wedge_{\TT}$. By
the minimality of $\boldsymbol{\curlywedge}^{\wedge_{\TT}}$ and
the maxi\-mality of $\blacktriangle^{\wedge_{\TT}}$, we
immediately obtain that
$\boldsymbol{\curlywedge}^{\wedge_{\TT}}\leq
{{\bigwedge}}_{\TT[X]} \leq \blacktriangle^{\wedge_{\TT}}$.
\end{proof}


\begin{remark} \rm
Given a semistar operation $\star$ of an integral domain $D$ which
is not a field, let $\TT := \{D_Q \mid Q \in \QMax^{\stf}(D)\}$.
Then ${\wedge_{\TT}} = \widetilde{\star}$ and
${\bigwedge_{\TT[X]}}=[\stt]$,  and hence by Theorem
\ref{ok-stable}, we have
$\boldsymbol{\curlywedge}^{\wedge_{\TT}}=\boldsymbol{\curlywedge}^{\stt}=[\stt]={\bigwedge_{\TT[X]}}
\lneq  \blacktriangle^{\stt} =\blacktriangle^{\wedge_{\TT}} $.
 For the general case, i.e., for an arbitrary family of overrings $\TT$,
 it would be interesting to know under which conditions $\boldsymbol{\curlywedge}^{\wedge_{\TT}}$ coincides with
$ {{\bigwedge}}_{\TT[X]}$.
\end{remark}

\smallskip

For a given semistar operation $\star$, we investigate the
relationship among the following semistar operations:
\begin{center}
$ \boldsymbol{\curlywedge}^{\star_a} \,,\;\;\;
\blacktriangle^{\star_a}\,,\;\;\;
(\boldsymbol{\curlywedge}^{\star})_a\,,\;\;\;
(\blacktriangle^{\star})_a\,,\;\;\;  {\bigwedge_{\VV(\star)[X]}} $
\end{center}
where $\VV(\star)$ is the family of all $\star$--valuation
overrings of $D$. Recall that, when $\star=d_D$
and  $\VV := \VV(d_D)$ is the family of all valuation overrings of $D$, then
 $ (d_D)_a =b_D = \bigwedge_{\VV}$.  Set $[b_D] :=  \bigwedge_{\VV[X]}$.
 Then, from Proposition \ref{prop23}, we immediately deduce that:

\begin{corollary} \label{cor25} For any integral domain $D$,
 $$ \boldsymbol{\curlywedge}^{b_D} \leq [b_D]\leq \blacktriangle^{b_D}\,.$$
\end{corollary}

\smallskip

For tackling the general question, we  need the following lemma.

\begin{lemma} \label{star-a}
Let $\star$ be a semistar operation of $D$ and let $\bigstar$ be a
strict extension of $\star$ to $D[X]$. Then $\bigstar_a$ is a
strict extension of $\star_a$.
\end{lemma}
\begin{proof} We start by proving some results of independent interest.

\noindent{\bf Claim 1.} {\sl If $H$ is a nonzero finitely
generated integral ideal of $D[X]$, then}
 $$\sum_{g\in H} \left(\co_D(g)\right)^r=
 \left(\sum_{g\in H} \co_D(g)\right)^{\!\! r} \; \mbox{ \sl for all }
r\geq 1\,.
$$

The inclusion $(\subseteq)$ is obvious. The opposite inclusion
$(\supseteq)$ follows from the observation that for an arbitrary choice of $g_1, g_2,
\cdots, g_m\in H$, $\co_D(g_1)+\co_D(g_2)+\cdots +\co_D(g_m)=\co_D(g)$ for some
$g\in H$ (we can put $g:=g_1+X^{\text{deg}(g_1)+1}g_2+
X^{\text{deg}(g_1)+\text{deg}(g_2)+2}g_3+ \cdots
+X^{\text{deg}(g_1)+\text{deg}(g_2)+\cdots
+\text{deg}(g_{m-1})+m-1}g_m$).


\noindent{\bf Claim 2.} {\sl If $E$ and $H$ are nonzero finitely
generated integral ideals of $D$  and $D[X]$, respectively,
then
$$((E[X]H)^{\bigstar}: H^{\bigstar})\subseteq E^{\star_a}[X].$$}

Note first  that  $((E[X]H)^{\bigstar}:
H^{\bigstar})\subseteq K[X]$.  Indeed, since  $HK[X]=hK[X]$
for some $h\in K[X]$, we have
$$
\begin{array}{rl}
((E[X]H)^{\bigstar}: H^{\bigstar})\subseteq &
((E[X]HK[X])^{\bigstar}: (HK[X])^{\bigstar}) \subseteq
((HK[X])^{\bigstar}: (HK[X])^{\bigstar})  \\ = & ((hK[X])^{\bigstar}:
(hK[X])^{\bigstar}) = ((K[X])^{\bigstar}: (K[X])^{\bigstar}) \\
= &(K^{\star}[X]:K^{\star}[X])\\ =&(K[X]:K[X])=K[X]\,.
\end{array}
$$

Let $f\in ((E[X]H)^{\bigstar}: H^{\bigstar}) \subseteq K[X]$. Then,
$$
fH\subseteq (E[X]H)^{\bigstar}\subseteq
(E[X]\co_D(H)[X])^{\bigstar}=(E\co_D(H))^{\star}[X].
$$
Let $m:=\text{deg}(f)$ and let $g\in H$. Then, by the previous
observation,
$$\co_D(f)\co_D(g)^{m+1}=\co_D(fg)\co_D(g)^m\subseteq
(E\co_D(H))^{\star}\co_D(H)^m\subseteq (E\co_D(H)^{m+1})^{\star},$$
and
so $\co_D(f)(\sum_{g\in H} \co_D(g)^{m+1})\subseteq
(E\co_D(H)^{m+1})^{\star}$.   By Claim 1, we deduce that
$$\co_D(f)\co_D(H)^{m+1}\subseteq (E\co_D(H)^{m+1})^{\star}.$$
Therefore, $\co_D(f)\subseteq ((E\co_D(H)^{m+1})^{\star}:
(\co_D(H)^{m+1})^{\star})$. Since $\co_D(H)^{m+1}$ is a finitely
generated ideal of $D$, $((E\co_D(H)^{m+1})^{\star}:
(\co_D(H)^{m+1})^{\star})\subseteq E^{\star_a}$. Thus we deduce
that $f\in E^{\star_a}[X]$.

\smallskip

From Claim 2, it easily follows that for each $E\in \FF(D)$,
$(E[X])^{\bigstar_a}\subseteq E^{\star_a}[X]$. Since the opposite
inclusion is obvious, the proof is completed.
\end{proof}

\begin{proposition}
Let $\star$ be a semistar operation of $D$. Then
$$\boldsymbol{\curlywedge}^{\star_a}\leq
(\boldsymbol{\curlywedge}^{\star_a})_a \leq
(\boldsymbol{\curlywedge}^{\star})_a \leq
(\blacktriangle^{\star})_a \leq (\blacktriangle^{\star_a})_a \lneq
\blacktriangle^{\star_a}.$$
\end{proposition}
\begin{proof} By Lemma \ref{star-a}, we have
$\boldsymbol{\curlywedge}^{\star_a}\leq
(\boldsymbol{\curlywedge}^{\star})_a$, and  since
$\boldsymbol{\curlywedge}^{\star_a}$ is of finite type,
$\boldsymbol{\curlywedge}^{\star_a}\leq
(\boldsymbol{\curlywedge}^{\star_a})_a$. Also, from the first
inequality, we have $(\boldsymbol{\curlywedge}^{\star_a})_a\leq
((\boldsymbol{\curlywedge}^{\star})_a)_a=(\boldsymbol{\curlywedge}^{\star})_a$.
Thus, we get $\boldsymbol{\curlywedge}^{\star_a}\leq
(\boldsymbol{\curlywedge}^{\star_a})_a \leq
(\boldsymbol{\curlywedge}^{\star})_a$. Similarly,
$(\blacktriangle^{\star})_a \lneq \blacktriangle^{\star_a}$ and
$(\blacktriangle^{\star_a})_a \lneq \blacktriangle^{\star_a}$,
where the strict inequalities follow from Proposition
\ref{prop10}. Moreover, from the first inequality, we also have
$(\blacktriangle^{\star})_a=((\blacktriangle^{\star})_a)_a \leq
(\blacktriangle^{\star_a})_a$. Thus, we get
$(\blacktriangle^{\star})_a \leq (\blacktriangle^{\star_a})_a
\lneq \blacktriangle^{\star_a}$. Finally, since
$\boldsymbol{\curlywedge}^{\star}\leq \blacktriangle^{\star}$,
obviously we have $(\boldsymbol{\curlywedge}^{\star})_a \leq
(\blacktriangle^{\star})_a$.
\end{proof}

\begin{remark}\label{remark5.7} \rm
(a)\, It  can happen  that
$\boldsymbol{\curlywedge}^{\star_a}\lneq
(\boldsymbol{\curlywedge}^{\star_a})_a$, i.e., in general,
$\boldsymbol{\curlywedge}^{\star_a}$ is not an \texttt{eab}
semistar operation. For instance, let $D$ be a Pr\"{u}fer domain,
not a field, and let $\star=d_D = b_D$. Then
$\boldsymbol{\curlywedge}^{\star_a} =
\boldsymbol{\curlywedge}^{b_D}=
\boldsymbol{\curlywedge}^{d_D}=d_{D[X]}\neq b_{D[X]}= (d_{D[X]})_a
= (\boldsymbol{\curlywedge}^{\star_a})_a$.

(b)\, It can happen that $(\boldsymbol{\curlywedge}^{\star})_a
\lneq (\blacktriangle^{\star})_a$. For instance, if $D$ is not a
field, then
$b_{D[X]}=(d_{D[X]})_a=(\boldsymbol{\curlywedge}^{d_D})_a\lneq
(\blacktriangle^{d_D})_a$. Indeed, let $\alpha$ be a nonzero
nonunit element of $D$ and let $A:=(\alpha, X)D[X]$. Since $A$ is
a finitely generated integral ideal of $D[X]$,
$A^{(\blacktriangle^{d_D})_a}\supseteq
A^{\blacktriangle^{d_D}}=D[X]$ (see Remark \ref{rem-ex}(c)). On
the other hand, recall that $A^{b_{D[X]}}=\bigcap AW$, where $W$
ranges over the valuation overrings of $D[X]$, and hence that
$A^{b_{D[X]}}\subseteq
(D[X])^{b_{D[X]}}=D^{b_D}[X]=\overline{D}[X]$, where
$\overline{D}$ is the integral closure of $D$. Let $N$ be a
maximal ideal of $\overline{D}$ containing $\alpha$. Then, $N+(X)$
is a prime ideal of $\overline{D}[X]$. By \cite[Theorem 19.6]{G},
there exists a valuation overring $W$ of $\overline{D}[X]$ such
that $N+(X)$ is the center of $W$ on $\overline{D}[X]$. Hence
$A^{b_{D[X]}}\subseteq AW\cap \overline{D}[X]\subseteq
(N+(X))W\cap \overline{D}[X]=N+(X)$. Therefore, $b_{D[X]}\neq
(\blacktriangle^{d_D})_a$.

(c)\, If $\star=\star_f$, then
$(\boldsymbol{\curlywedge}^{\star_a})_a=
(\boldsymbol{\curlywedge}^{\star})_a$. Because, if
$\star=\star_f$, then $\star\leq \star_a$ and hence
$\boldsymbol{\curlywedge}^{\star}\leq
\boldsymbol{\curlywedge}^{\star_a}$. Consequently,
$(\boldsymbol{\curlywedge}^{\star})_a\leq
(\boldsymbol{\curlywedge}^{\star_a})_a$ and hence the equality
holds. However, we do not know whether it is possible in general
that $(\boldsymbol{\curlywedge}^{\star_a})_a \lneq
(\boldsymbol{\curlywedge}^{\star})_a$. This problem is related
with the inequality $\boldsymbol{\curlywedge}^{\star_f} \leq
(\boldsymbol{\curlywedge}^{\star})_f$. If
$\boldsymbol{\curlywedge}^{\star_f} =
(\boldsymbol{\curlywedge}^{\star})_f$, then we have
$(\boldsymbol{\curlywedge}^{\star_a})_a
=(\boldsymbol{\curlywedge}^{(\star_f)_a})_a=
(\boldsymbol{\curlywedge}^{\star_f})_a
=((\boldsymbol{\curlywedge}^{\star})_f)_a=
(\boldsymbol{\curlywedge}^{\star})_a$. As mentioned in
Remark~\ref{remark4.8}, we do not know whether the equality
$\boldsymbol{\curlywedge}^{\star_f} =
(\boldsymbol{\curlywedge}^{\star})_f$ holds or not.

(d)\, Without much difficulty, we can show that the set of
$\blacktriangle^{\star_a}$-valuation overrings of $D[X]$ is the
set $\{K[X]_{(f)}\mid f \mbox{ is an irreducible polynomial of }
K[X]\} \, \bigcup \, \{ V(X)\mid V \mbox{ is a $\star$-valuation
 overring of }$ $ D \}$. However, we do not have any information about the
$\blacktriangle^{\star}$-valuation overrings of $D[X]$, and thus
we do not know whether it is possible that
$(\blacktriangle^{\star})_a \lneq (\blacktriangle^{\star_a})_a$.

\end{remark}

\begin{corollary}
 Let $D$ be an integral domain with quotient field $K$. Then,
$$\boldsymbol{\curlywedge}^{b_D}\leq [b_D] \leq b_{D[X]}=
(\boldsymbol{\curlywedge}^{b_D})_a = [b_D]_a \leq
(\blacktriangle^{b_D})_a\lneq \blacktriangle^{b_D}.$$ Moreover,
$b_{D[X]}=(\blacktriangle^{b_D})_a$ if and only if $D=K$.
\end{corollary}
\begin{proof}
By Remark~\ref{remark5.7}(c),
$b_{D[X]}=(d_{D[X]})_a=(\boldsymbol{\curlywedge}^{d_D})_a
=(\boldsymbol{\curlywedge}^{b_D})_a$, and by \cite[Proposition
15]{CF2011}, $b_{D[X]} = [b_D]_a$. In order to show that
$[b_D]\leq b_{D[X]}$, let $E\in \FF(D[X])$ and let $W$ be a
valuation overring of $D[X]$. Then $V:=W\cap K$ is a valuation
overring of $D$, and hence $EW\supseteq EV[X]\supseteq E^{[b_D]}$.
Therefore, $E^{b_{D[X]}}\supseteq E^{[b_D]}$.

If $D=K$, then by Remark~\ref{rem-ex}(c),
$d_{D[X]}=(\blacktriangle^{d_D})_f$, and hence
$b_{D[X]}=(d_{D[X]})_a=((\blacktriangle^{d_D})_f)_a=
(\blacktriangle^{d_D})_a=(\blacktriangle^{b_D})_a$. If $D\neq K$,
then $b_{D[X]}\lneq (\blacktriangle^{d_D})_a$ (see
Remark~\ref{remark5.7}(b)). Since $(\blacktriangle^{d_D})_a \leq
(\blacktriangle^{b_D})_a$, it immediately follows that
$b_{D[X]}\lneq (\blacktriangle^{b_D})_a$.
\end{proof}

\begin{remark}\rm
It can happen that $[b_D]\lneq b_{D[X]}$. For instance, let $D$ be
a Pr\"{u}fer domain which is not a field. Then
$[b_D]=[d_D]=d_{D[X]}\lneq b_{D[X]}$.
\end{remark}

\bigskip



\end{document}